\documentclass[]{interact}

\usepackage{epstopdf}\usepackage[caption=false]{subfig}

\usepackage[numbers,sort&compress]{natbib}
\bibpunct[, ]{[}{]}{,}{n}{,}{,}
\renewcommand\bibfont{\fontsize{10}{12}\selectfont}
\theoremstyle{plain}
\newtheorem{theorem}{Theorem}[section]

\newtheorem{corollary}[theorem]{Corollary}
\newtheorem{proposition}[theorem]{Proposition}

\theoremstyle{definition}
\newtheorem{definition}[theorem]{Definition}

\newtheorem{subproblem}[theorem]{Subproblem}%

\theoremstyle{remark}
\newtheorem{remark}{Remark}

\newtheorem{assumption}{Assumption}%

\usepackage{graphicx}%
\usepackage{multirow}%
\usepackage{amsmath,amssymb,amsfonts}%
\usepackage{amsthm}%
\usepackage{mathrsfs}%
\usepackage[title]{appendix}%
\usepackage{xcolor}%
\usepackage{textcomp}%
\usepackage{manyfoot}%
\usepackage{booktabs}%
\usepackage{algorithm}%
\usepackage{algorithmicx}%
\usepackage{algpseudocode}%
\usepackage{listings}%

\usepackage{bm} 

\newcommand{\bx}{\bm{x}}
\newcommand{\y}{\bm{y}}

\newcommand{\dd}{\bm{d}}

\newcommand{\z}{\bm{z}}

\newcommand{\SSS}{\mathcal{S}}
\newcommand{\T}{\mathcal{T}}

\newcommand{\W}{\bm{W}}

\newcommand{\aaa}{\bm{a}}
\newcommand{\bbb}{\bm{b}}

\usepackage{setspace}

\usepackage{varwidth}
\usepackage{mdframed}

\usepackage{tikz}

\usepackage{hyperref}

\begin{document}

\articletype{}

\title{A New Two-dimensional Model-based Subspace Method for Large-scale Unconstrained Derivative-free Optimization: \texttt{2D-MoSub}}

\author{
\name{Pengcheng Xie\textsuperscript{a,b}\thanks{Pengcheng Xie. Email: xpc@lsec.cc.ac.cn} and Ya-xiang Yuan\textsuperscript{a,b}\thanks{Ya-xiang Yuan. Email: yyx@lsec.cc.ac.cn}}
\affil{\textsuperscript{a}State Key Laboratory of Scientific/Engineering Computing, Institute of Computational Mathematics and Scientific/Engineering Computing, Academy of Mathematics and Systems Science, Chinese Academy of Sciences, Zhong Guan Cun Donglu No.55, Beijing, China; \textsuperscript{b}University of Chinese Academy of Sciences, Zhong Guan Cun Donglu No.55, Beijing, China}
}

\maketitle

\begin{abstract}
This paper proposes the method \texttt{2D-MoSub} (2-dimensional model-based subspace method), which is a novel derivative-free optimization (DFO) method based on the subspace method for general unconstrained optimization and especially aims to solve large-scale DFO problems. Our method combines 2-dimensional quadratic interpolation models and trust-region techniques to iteratively update the points and explore the 2-dimensional subspace. Its framework includes initialization, constructing the interpolation set, building the quadratic interpolation model, performing trust-region trial steps, and updating the trust-region radius and subspace. We introduce the framework and computational details of \texttt{2D-MoSub}, and discuss the poisedness and quality of the interpolation set in the corresponding 2-dimensional subspace. We also analyze some properties of our method, including the model's approximation error with projection property and the algorithm's convergence. Numerical results demonstrate the effectiveness and efficiency of \texttt{2D-MoSub} for solving a variety of unconstrained optimization problems.
\end{abstract}


\section{Introduction}

Optimization problems arise in numerous fields, and derivative-free optimization methods have gained attention due to their applicability in scenarios where derivative information is either unavailable or expensive to compute. For more details on derivative-free optimization, one can see the book of Conn, Scheinberg and Vicente \cite{conn2009introduction}, Audet and Hare's book \cite{audet2017derivative}, the survey paper of Larson, Menickelly and Wild \cite{larson_menickelly_wild_2019}, and Zhang's survey \cite{zhang2021}.

Model-based trust-region methods are well-established algorithms for nonlinear programming problems. Some of them are known for the reliance on quadratic models, which enable them to handle curvature information effectively. By minimizing quadratic models over trust regions, trust-region methods iteratively solve the optimization problems \cite{Winfield1969, Mattias2000,Powell2003,Conn2009a,Gratton-2017, cartis2019improving,xie2023h2}.

For existing techniques for model-based derivative-free optimization, the large-scale problems are still bottlenecks since the computational cost and interpolation error of constructing local (polynomial) models can be high when the dimension of the problem is high. One can regard this as the curse of dimensionality in the derivative-free optimization. Traditional derivative-free optimization methods use quadratic interpolation models to approximate the objective function, but the limited available information results in the low credibility of these models. Some methods have been developed to handle large-scale problems. One important approach is the use of subspace methods, which involves minimizing the objective function in a low-dimensional subspace at each iteration to obtain the next iteration point \cite{cartis2023scalable,yuan2007subspace,yuan2009subspace,yuan2014review}. 

\begin{sloppypar}
To solve the (large-scale) unconstrained derivative-free optimization problem
\[
\min _{\bx \in \Re^n}\ f(\bx),
\] 
this paper introduces a new derivative-free optimization method for large-scale black-box problems that leverages subspace techniques and quadratic interpolation models to efficiently search for the optimal solution. Our new method \texttt{2D-MoSub} holds a skeleton in the subspace method's type and uses 2-dimensional quadratic models in 2-dimensional subspaces iteratively. The released codes can be downloaded from the online repository\footnote{\href{https://github.com/PengchengXieLSEC/Large-Scale}{\texttt{https://github.com/PengchengXieLSEC/Large-Scale}}}. 
\end{sloppypar}

{\bf Organization.} The rest of this paper is organized as follows. Section \ref{section:2D-MoSub} introduces the framework and computational details of the new DFO method \texttt{2D-MoSub} and the coordinate transformation referring to the subspace. Section \ref{section:Interpolation set's poisedness and interpolation quality} discusses the poisedness and quality of the interpolation set. Section \ref{section:Some properties of 2D-MoSub} analyzes some properties of \texttt{2D-MoSub}, {including the approximation error with projection property and the convergence}. Numerical results for solving large-scale problems are shown in Section \ref{section:Numerical results}.

\section
{\texttt{2D-MoSub}}

\label{section:2D-MoSub}

In this section, we will briefly introduce the proposed method, \texttt{2D-MoSub}. Its framework is shown in Algorithm \ref{A new subspace derivative-free optimization method}.  

\begin{algorithm}[htbp]
 \fontsize{8}{7}\selectfont
\caption{Two-dimensional Model-based Subspace Method (\texttt{2D-MoSub})\label{A new subspace derivative-free optimization method}} 
\begin{algorithmic}[1]
  \setstretch{0.1} 

\State {\bf Input.} \(\bx_0 \in \Re^n, \Delta_1, \Delta_{\text{low}}, \gamma_{1},\gamma_{2},\eta,\eta_0\), \(\dd^{(1)}\in \Re^n\), \(k=1\).  

\State {\bf Step 0. (Initialization)} 

Obtain \(\y_{a},\y_{b},\y_{c}\) and \(\dd_1^{(1)}\). Construct the initial 1-dimensional quadratic model \(Q^{\text{sub}}_1\) on the 1-dimensional space \(\bx_{1}+\text{span}\{\dd_{1}^{(1)}\}\).

\State {\bf Step 1. (Constructing the interpolation set)}

{Obtain \(\dd_2^{(k)}\in \Re^n\) such that \(\left\langle\dd_1^{(k)}, \dd_2^{(k)}\right\rangle=0\). Obtain \(\y_1^{(k)},\y_2^{(k)},\y_3^{(k)}\).}

\State {\bf Step 2. (Constructing the  quadratic interpolation model)}

Construct the 2-dimensional quadratic model \(Q_k\) on the 2-dimensional space \(\bx_k+\text{span}\{\dd_{1}^{(k)}, \dd_{2}^{(k)}\}\).

\State {\bf Step 3. (Trust-region trial step)}

Solve the trust-region subproblem of \(Q_k\) and selectively solve the trust-region subproblem of a modified model \(Q_k^{\rm mod}\), and then obtain \({\bx}^{+}_k\). Compute 
\begin{equation}
\label{ratioofsuccessfulstep}
\rho_k=\frac{f({\bx}^{+}_k)-f(\bx_{k})}{Q_k({\bx}^{+}_k)-Q_k(\bx_{k})}. 
\end{equation} 
Update and obtain \({\bx}_{k+1}\) and \(\dd_1^{(k+1)}\). Go to {\bf Step 4}.

\State {\bf Step 4. (Updating)}

\begin{sloppypar}
If \(\Delta_k<\Delta_{\text{low}}\), then terminate. Otherwise, update \(\Delta_{k+1}\), and 
let 
\[
\dd_{1}^{(k+1)}=\frac{\bx_{k+1}-\bx_{k}}{\Vert \bx_{k+1}-\bx_{k} \Vert_2},
\] 
and construct \(Q^{\text{sub}}_{k+1}\) as the function \(Q_k\) on the 1-dimensional space \(\bx_{k+1}+\text{span}\{\dd_{1}^{(k+1)}\}\). 
Increment \(k\) by one and go to {\bf Step 1}.
\end{sloppypar}

\end{algorithmic} 
\end{algorithm}

We will introduce the details of each part of Algorithm \ref{A new subspace derivative-free optimization method} in the following respectively. Before further discussion, we give the following remark and definition. 

\begin{remark}
\begin{sloppypar}
In our method, the points in \({\Re}^n\) will be selectively regarded as the points in a 1-dimensional subspace or the points in a 2-dimensional subspace when we discuss the corresponding 1-dimensional interpolation model and 2-dimensional interpolation model. 
\end{sloppypar}
\end{remark}

Therefore, we define the following transformations to achieve the coordinate transformation from the 2-dimensional subspace \(\SSS_{\aaa,\bbb}^{(k)}=\bx_k+\operatorname{span}\left\{\aaa, \bbb\right\}\) and the 1-dimensional subspace \(\hat{\mathcal{S}}_{\aaa}^{(k)}=\bx_k+\operatorname{span}\left\{\aaa\right\}\) to \(\Re^2\) or \(\Re\) separately, given \(\aaa\in\Re^{n}\) and \(\bbb\in\Re^{n}\).

\begin{definition}

Let \(\T_{\aaa,\bbb}^{(k)}\) be the transformation from \(\SSS_{\aaa,\bbb}^{(k)}\) to \(\Re^2\), defined by  
\begin{equation*}
\T_{\aaa,\bbb}^{(k)}: \ \y  \mapsto\left(\left\langle \y-\bx_k, \aaa\right\rangle, \left\langle \y-\bx_k, \bbb\right\rangle\right)^{\top}.
\end{equation*} 
Let \(\hat{\T}_{\aaa}^{(k)}\) be the transformation from \(\hat{\SSS}_{\aaa}^{(k)}\) to \(\Re\), defined by  
\(
 \hat{\T}_{\aaa}^{(k)}: \ 
 \y \mapsto\left\langle \y-\bx_k, \aaa\right\rangle.
\)
\end{definition}

\subsection{Initialization}

The algorithm \texttt{2D-MoSub} starts by initializing the input parameters and vectors. It first constructs the initial 1-dimensional quadratic interpolation model \(Q_1^{\text{sub}}\). To initiate the algorithm, we begin by initializing it with the initial point $\bm{x}_0$ and setting various parameters: trust-region parameters $\Delta_1$ and $\Delta_{\text{low}}$, $\gamma_{1}$ and $\gamma_{2}$, as well as successful step thresholds $\eta$ and $\eta_0$. $\bm{d}^{(1)}$ can be any direction in \(\Re^{n}\). For example, we can let $\bm{d}^{(1)}=(1,0,\dots,0)^{\top}$. A large number of numerical experiments show that the choice of \(\bm{d}^{(1)}\) does not have an essential influence.

\begin{figure*}[htb]
\begin{mdframed}[
    linewidth=0.8pt,
    hidealllines=true,
    topline=true,
    bottomline=true,
    frametitle={Step 0. (Initialization)}] 
    \fontsize{8}{7}\selectfont
      \setstretch{0.1} 
\begin{algorithmic}[1]
\State \textbf{Input:} Initialize  with the initial point $\bm{x}_0$, trust-region parameters $\Delta_1$, $\Delta_{\text{low}}$, $\gamma_{1}$, $\gamma_{2}$, and thresholds $\eta$ and $\eta_0$. Choose a $\bm{d}^{(1)}\in \Re^{n}$.
\State Obtain three points: $\bm{y}_a = \bm{x}_0$, $\bm{y}_b = \bm{x}_0 + \Delta_1 \bm{d}^{(1)}$, and $\bm{y}_c$ based on the relative values of $f(\bm{y}_a)$ and $f(\bm{y}_b)$, i.e.,
\begin{equation}
\label{yc1}
\y_{c}=\left\{
\begin{aligned}
&\y_{a} + 2\Delta_1 \dd^{(1)}, \text{ if }f(\y_{a})\le f(\y_{b}),\\
&\y_{a} - {\Delta_1} \dd^{(1)}, \text{ otherwise}. 
\end{aligned}
\right.
\end{equation}
\State Let $\bm{x}_{1}$ be the minimizer of $f(\bm{y})$ among $\bm{y}_a$, $\bm{y}_b$, and $\bm{y}_c$:
\[
\bx_{1}=\arg\min_{\y\in\{\y_{a},\y_{b},\y_{c}\}}\ f(\y).
\]

\State Let $\bm{y}_{\max,1}^{(1)}$ be the maximizer of $f(\bm{y})$ among $\bm{y}_a$, $\bm{y}_b$, and $\bm{y}_c$:
\[
\y_{\max,1}^{(1)}=\underset{\y\in\{\y_{a},\y_{b},\y_{c}\}}{\arg\max} f(\y). 
\]

\State Determine $\bm{d}^{(1)}_1$ as the normalized vector from $\bm{x}_1$ to $\bm{y}_{\max,1}^{(1)}$:
\[
\dd^{(1)}_1=\frac{\bx_1-\y_{\max,1}^{(1)}}{\Vert \bx_1-\y_{\max,1}^{(1)}\Vert_2}.
\]

\State Construct the initial 1-dimensional quadratic interpolation model \(Q^{\text{sub}}_1\) on the 1-dimensional space \(\bx_{1}+\text{span}\{\dd_{1}^{(1)}\}\) as
\begin{equation}
\label{Q1sub}
Q_1^{\text{sub}}(\alpha)=f(\bx_1)+a^{(1)} \alpha +b^{(1)} \alpha^2,
\end{equation}
where \(a^{(1)}, b^{(1)} \in \Re\) are determined by the interpolation conditions 
\begin{equation}
\label{Q1subinter}
Q_1^{\text{sub}}(\hat{\T}_{\dd_1^{(1)}}^{(1)}(\y))=f(\y), \ \forall \ \y \in \{\y_{a},\y_{b},\y_{c}\}.
\end{equation}

\end{algorithmic}
\end{mdframed}
\end{figure*}

Once the parameters' initialization above is complete, \texttt{2D-MoSub} obtains three points: $\bm{y}_a = \bm{x}_0$, $\bm{y}_b = \bm{x}_0 + \Delta_1 \bm{d}^{(1)}$, and $\bm{y}_c$, which is determined by (\ref{yc1}) based on the relative values of $f(\bm{y}_a)$ and $f(\bm{y}_b)$. With the points above established, \texttt{2D-MoSub} finds $\bm{x}_{1}$, which serves as the minimizer of $f$ among $\bm{y}_a$, $\bm{y}_b$, and $\bm{y}_c$. Simultaneously, \texttt{2D-MoSub} calculates $\bm{y}_{\max,1}^{(1)}$, which represents the maximizer of $f$ among the same set of points $\bm{y}_a$, $\bm{y}_b$, and $\bm{y}_c$. Notice that we will make sure that $\bm{x}_1\ne \bm{y}_{\max,1}^{(1)}$ in the case where they share the same function value. Using $\bm{x}_1$ and $\bm{y}_{\max,1}^{(1)}$, we determine $\bm{d}^{(1)}_1$ as the normalized vector pointing from $\bm{x}_1$ to $\bm{y}_{\max,1}^{(1)}$. Finally, \texttt{2D-MoSub} constructs the initial 1-dimensional quadratic interpolation model $Q^{\text{sub}}_1$ as (\ref{Q1sub}) on the 1-dimensional space $\bm{x}_{1}+\text{span}\{\bm{d}_{1}^{(1)}\}$ according to the interpolation condition (\ref{Q1subinter}). The details are shown in the algorithmic pseudo-code of the initialization step numerated as Step 0 of \texttt{2D-MoSub}.

\subsection{Constructing the interpolation set}
Before achieving the interpolation in the 2-dimensional subspace, a unit vector orthogonal to the previously used direction \(\dd_1^{(k)}\) is chosen by \texttt{2D-MoSub} to construct the interpolation set. \texttt{2D-MoSub} evaluates the objective function at selected points. After selecting a unit vector $\bm{d}_2^{(k)}$ orthogonal to $\bm{d}_1^{(k)}$, we set $\bm{y}_1^{(k)}$, $\bm{y}_2^{(k)}$, $\bm{y}_3^{(k)}$ as (\ref{y1k}), (\ref{y2k}) and (\ref{y3k}). Details are shown in the following algorithmic pseudo-code of the interpolation set constructing step numerated as Step 1 of \texttt{2D-MoSub}. The unit vector \(\dd_2^{(k)}\) is determined randomly. Fig. \ref{initial_case_fig} shows the initial case and the subspace \(\bm{x}_{1}+\text{span}\{\dd_1^{(1)},\dd_2^{(1)}\}\), in which it illustrates the case when \(\y_{\text{min},2}^{(1)}=\y_{1}^{(1)}\).

\begin{figure*}[htbp]
\begin{center}
\begin{mdframed}[
    linewidth=0.8pt,
    hidealllines=true,
    topline=true,
    bottomline=true,
    frametitle={Step 1. (Constructing the interpolation set)}]
 \fontsize{8}{7}\selectfont
      \setstretch{0.1} 
\begin{algorithmic}[1]
 \State \textbf{Input:}  $\bm{x}_k$, $\bm{d}_1^{(k)}$, $\Delta_k$
    
    \State {\color{black}Choose a unit vector $\bm{d}_2^{(k)}$ such that \(\left\langle\dd_1^{(k)}, \dd_2^{(k)}\right\rangle=0\).}

    \State Set \(\y_1^{(k)}\) by
     \begin{equation}
     \label{y1k}
     \bm{y}_1^{(k)} = \bm{x}_k + \Delta_k \bm{d}_2^{(k)}.
     \end{equation}
    
    \State Determine $\bm{y}_2^{(k)}$ based on the relative values of $f(\bm{y}_1^{(k)})$ and $f(\bm{x}_k)$: 
    \begin{equation}
    \label{y2k}
\y_2^{(k)}=
\left\{
\begin{aligned}
&\bx_k+ 2\Delta_k \dd_2^{(k)}, \text{ if } f(\y_1^{(k)})\le f(\bx_k),\\
&\bx_k - \Delta_k \dd_2^{(k)}, \text{ otherwise}.
\end{aligned}
\right.
    \end{equation}
    
    \State Find $\bm{y}_{\min,2}^{(k)}$ as the minimizer of $f$ among $\bm{y}_1^{(k)}$ and $\bm{y}_2^{(k)}$:
    \[
\y_{\min,2}^{(k)}=\arg\min_{\y\in\{\y_1^{(k)},\y_2^{(k)}\}}\ f(\y). 
\]
    
    \State Set \(\bm{y}_{3}^{(k)}\) by 
\begin{equation}
\label{y3k}
\bm{y}_{3}^{(k)} = \bm{y}_{\min,2}^{(k)} + \Delta_k \bm{d}_1^{(k)}. 
\end{equation}

\end{algorithmic}
\end{mdframed}
\end{center}
\end{figure*}

\begin{figure}[H]
\centering 
\fbox{
\begin{tikzpicture}[scale=0.8]

\coordinate (ya) at (2,0);
\coordinate (yb) at (0,0);
\coordinate (yc) at (4,0);
\coordinate (y1) at (4,2);
\coordinate (y2) at (4,-2);
\coordinate (y3) at (6,2);

\draw (ya) circle [radius=2.5pt] node[above left] {$\bm{x}_{0}$};
\fill (yb) circle [radius=0pt] node[above left] {$\bm{y}_{\max,1}^{(1)}$};
\draw (yc) circle [radius=2.5pt] node[above left] {$\bm{x}_{1}$};
\fill (ya) circle [radius=1.5pt] node[below left] {$\bm{y}_a$};
\fill (yb) circle [radius=1.5pt] node[below left] {$\bm{y}_b$};
\fill (yc) circle [radius=1.5pt] node[below left] {$\bm{y}_c$};
\fill (y1) circle [radius=1.5pt] node[below left] {$\bm{y}_1^{(1)}$};
\fill (y2) circle [radius=1.5pt] node[above left] {$\bm{y}_2^{(1)}$};
\fill (y3) circle [radius=1.5pt] node[below left] {$\bm{y}_3^{(1)}$};

\draw[dashed] (yb) -- (ya) -- (yc);
\draw[dashed] (y2) -- (y1) -- (y3);

\draw[->,thick] (yc) -- (5,0) node[below left] {$\bm{d}_1^{(1)}$};
\draw[->,thick] (yc) -- (4,1) node[below right] {$\bm{d}_2^{(1)}$};


\end{tikzpicture}
}
\caption{The initial case and the subspace $\bm{x}_{1}+\text{span}\{\dd_1^{(1)},\dd_2^{(1)}\}$\label{initial_case_fig}}
\end{figure}

\subsection{Constructing the quadratic interpolation models}

A 2-dimensional quadratic interpolation model \(Q_k\) will be constructed based on the interpolation set in the 2-dimensional subspace \(\bx_{k}+\text{span}\{\dd_1^{(k)},\dd_2^{(k)}\}\) and the 1-dimensional quadratic interpolation model \(Q_{k}^{\text{sub}}\). The model \(Q_k\) will approximate the objective function in the current 2-dimensional subspace spanned by the selected directions. In addition, if the iteration point given by minimizing the model \(Q_k\) in the trust region does not achieve a sufficiently low function value, our method will construct the modified model \(Q_k^{\rm mod}\) based on the existed evaluated points. We will separately give the details of constructing the models.

\subsubsection[Obtaining Qk based on Qksub]{Obtaining \(Q_k\) based on \(Q_k^{\text{sub}}\)}

After taking the discussion on the initial  1-dimensional model \(Q^{\text{sub}}_{1}\) as (\ref{Q1sub}) according to the interpolation conditions (\ref{Q1subinter}), we introduce the iterative way how we obtain the \(k\)-th model \(Q_k\) based on \(Q_k^{\text{sub}}\).

We construct the 2-dimensional quadratic interpolation model \(Q_k\) on the 2-dimensional subspace \(\bm{x}_k + \text{span}\{\bm{d}_1^{(k)}, \bm{d}_2^{(k)}\}\) as 
\begin{equation}
\label{model}
Q_k(\alpha,\beta) = f(\bm{x}_k) + a^{(k)} \alpha + b^{(k)} \alpha^2 + c^{(k)} \beta + d^{(k)} \beta^2 + e^{(k)} \alpha \beta,
\end{equation} 
where \(a^{(k)}\) and \(b^{(k)}\) are given by \(Q_k^{\text{sub}}\), and 
\(c^{(k)}, d^{(k)}\) and \(e^{(k)}\) are determined by the interpolation conditions  
\begin{equation}
\label{interpolation-cond}
Q_k(\T_{\dd_1^{(k)},\dd_2^{(k)}}^{(k)}(\z)) = f(\z), \ \forall \ \z \in \{\bm{y}_{1}^{(k)},\bm{y}_{2}^{(k)},\bm{y}_{3}^{(k)}\}. 
\end{equation}

Details are shown in the algorithmic pseudo-code of the quadratic interpolation model constructing step numerated as Step 2 of \texttt{2D-MoSub}.
\begin{figure*}[htbp]
\begin{center}
\begin{mdframed}[
    linewidth=0.8pt,
    hidealllines=true,
    topline=true,
    bottomline=true,
    frametitle={Step 2. (Constructing the quadratic interpolation model)}]
     \fontsize{8}{7}\selectfont
      \setstretch{0.1}  
\begin{algorithmic}[1]
  \State Construct the 2-dimensional quadratic interpolation model \(Q_k\) on the 2-dimensional space \(\bm{x}_k + \text{span}\{\bm{d}_1^{(k)}, \bm{d}_2^{(k)}\}\) as (\ref{model}), where \(c^{(k)}, d^{(k)},\) and \(e^{(k)}\) are determined by the interpolation conditions (\ref{interpolation-cond}).
\end{algorithmic}
\end{mdframed}
\end{center}
\end{figure*}

The way obtaining \(Q_k\) is reasonable and reliable since \(\bx_{k}-\bx_{k-1}\) provides a sufficiently good direction and corresponding 1-dimensional subspace numerically when \(\bx_{k}\) is a successful trust-region trial step, and the model \(Q_{k-1}^+\) in the previous 2-dimensional subspace is a sufficiently good approximation along such a 1-dimensional subspace as well. This 1-dimensional subspace is the intersection of the \((k-1)\)-th 2-dimensional subspace and the \(k\)-th 2-dimensional subspace in most cases. In other words, \(Q_{k}\) follows a kind of optimality agreement with the previous model in the corresponding 1-dimensional subspace.

\subsubsection[Obtaining Qk1sub based on Qk]{Obtaining \(Q_{k+1}^{\text{sub}}\) based on \(Q_{k}\)}
\label{conditionsofQk+section}

We now introduce how our algorithm obtains the \((k+1)\)-th model function on the 1-dimensional subspace, \(Q^{\text{sub}}_{k+1}\), based on the \(k\)-th model function \(Q_k\). 
At the \(k\)-th step, we already have the model function \(Q_k\) as (\ref{model}) with determined coefficients. 
However, after we obtain the iteration point \(\bx_{k+1}\) and the function value \(f(\bx_{k+1})\), it is usual that
\begin{equation*}
Q_k(\T_{\dd_1^{(k)},\dd_2^{(k)}}^{(k)}(\bx_{k+1}))\ne f(\bx_{k+1}),  
\end{equation*}
which indicates that \(Q_k\) is not a perfect interpolation of \(f\) in the interval between \(\bx_{k}\) and \(\bx_{k+1}\). Therefore, before obtaining the model \(Q_{k+1}^{\text{sub}}\), we first update \(Q_k\) to the following \(Q_k^{+}\), which is
\begin{equation*}
Q_k^{+}(\alpha,\beta)=f(\bx_{k+1})+\bar a^{(k)} \alpha+\bar{b}^{(k)} \alpha^2+\bar c^{(k)} \beta+\bar{d}^{(k)} \beta^2+\bar e^{(k)} \alpha \beta,
\end{equation*}
and it satisfies the interpolation conditions 
\begin{equation*}
Q_k^{+}(\T_{\dd_1^{(k+1)},\dd_*^{(k+1)}}^{(k+1)}(\bm{z}))=f(\bm{z}), \ \forall\ 
\z \in\left\{\bx_{k-1}, \bx_k, \bx_{k+1}, \y_1^{(k)}, \y_2^{(k)}, \y_3^{(k)}\right\},
\end{equation*}
where \(\dd_{*}^{(k+1)}\in \text{span}\{\dd_1^{(k)},\dd_2^{(k)}\}\) satisfies that \(\left\langle\dd_{*}^{(k+1)},\dd_{1}^{(k+1)}\right\rangle=0\). 
{
We show more details about constructing the quadratic interpolation model \(Q_k^{+}\) in the following. 

Notice that in some cases, the set of the interpolation points is not poised\footnote{The sense of poised (or poisedness) is as same as the one here when discussing the fully interpolation.} in the sense where the interpolation equations do not have a solution (do not provide a unique model). Therefore, \texttt{2D-MoSub} will check if the corresponding coefficient matrix of each system of interpolation equations is invertible or not. If the current interpolation set is not poised in the above sense, then the method \texttt{2D-MoSub} has different interpolation sets in reserve, and it will test the corresponding invertibility of the set and use a poised set to obtain the quadratic model \(Q_k^{+}\) by solving the  interpolation equations. The sets of interpolation points in reserve are formed by 6 points chosen from the set 
\[
\mathcal{Y}^+_k=\{\bx_{k-1},\bx_{k},\bx_{k+1},\y_{1}^{(k)},\y_{2}^{(k)},\y_{3}^{(k)},\y_{4}^{(k)},\y_{5}^{(k)}\},
\]
where 
\( 
\y^{(k)}_4=\bx_{k}+\frac{\sqrt{2}}{2}\Delta_k \dd_1^{(k)}+\frac{\sqrt{2}}{2}\Delta_k \dd_2^{(k)}, 
\) 
and 
\(   
\y^{(k)}_5=\bx_{k}+\Delta_k {\dd_1^{(k)}}. 
\)   
Notice that the points \(\bx_{k},\y_{1}^{(k)},\y_{2}^{(k)},\y_{3}^{(k)},\y_{4}^{(k)}\) and \(\y_{5}^{(k)}\) have a fixed distribution, and thus such a choice works. Therefore, we obtain the \(Q_{k}^+\) according to the interpolation conditions 
\begin{equation}
\label{Q+condconclude}
Q_k^{+}(\T_{\dd_1^{(k+1)},\dd_*^{(k+1)}}^{(k+1)}(\bm{z}))=f(\bm{z}), \ \forall\ 
\z \in \mathcal{Y},
\end{equation}
where \(\mathcal{Y}\subset \mathcal{Y}^+_k\) and \(\mathcal{Y}\) is a set with 6 poised interpolation points.

The modification above is necessary to make the interpolation set to be poised for constructing the quadratic model \(Q_k^{+}\). Thus, in all the cases, we can obtain the 2-dimensional quadratic model \(Q_k^{+}\). Then the 1-dimensional model \(Q_{k+1}^{\text{sub}}\) is set as  
\begin{equation}
\label{Qk+1subconditon}
Q_{k+1}^{\text{sub}}(\alpha)=Q_k^{+}(\alpha,0),\ \forall\ \alpha\in\Re. 
\end{equation}

\begin{sloppypar}
{\texttt{2D-MoSub} saves the computational cost since it does not need a programmed subroutine to provide the model-improvement step in the traditional algorithms.}
\end{sloppypar} 

\begin{remark}
There is also another way to obtain the modified model \(Q_k^{+}\) at the successful step, which follows the least norm updating way based on the updated interpolation set \(\mathcal{X}_{k}^{+}=\mathcal{X}_{k}\cup \{\bx_{k+1}\}\backslash \{\bx_{k-1}\}\), and we only need to solve the corresponding Karush–Kuhn–Tucker equations of the least norm updating model's subproblem.  
\end{remark}

\subsubsection[Obtaining Qkmod based on Qk]{Obtaining \(Q_{k}^{\rm mod}\) based on \(Q_{k}\)}

If the iteration point \(\bx_k^{+}\) given by solving the trust-region subproblem of the model function \(Q_k\) is not good enough in the sense of achieving the decrement of the function value in the case where \(\bx_{k}^{+}\notin\{\bx_{k-1}, \bx_k, \bm{y}_{1}^{(k)},\bm{y}_{2}^{(k)},\bm{y}_{3}^{(k)}\}\), our method will re-solve the trust-region subproblem of the modified quadratic model \(Q_k^{\rm mod}\). The modified model \(Q_k^{\rm mod}\) is constructed by the interpolation conditions 
\begin{equation}
\label{intercondQkmod}
Q_k^{\rm mod}(\T_{\dd_1^{(k)},\dd_2^{(k)}}^{(k)}(\z)) = f(\z), \ \forall \ \z \in \mathcal{Y}_k^{\rm mod},
\end{equation}
where
\[
\mathcal{Y}_k^{\rm mod}=
\left\{
\begin{aligned}
&\{\bx_{k-1}, \bx_k, \bx_{k}^{+}, \bm{y}_{1}^{(k)},\bm{y}_{2}^{(k)},\bm{y}_{3}^{(k)}\}, \ \text{if \(\bx_k\ne \bx_{k-1}\)},\\
&\{\bx_k, \bx_{k}^{+}, \bm{y}_{1}^{(k)},\bm{y}_{2}^{(k)},\bm{y}_{3}^{(k)}, \bm{y}_{4}^{(k)}\}, \ \text{if \(\bx_k=\bx_{k-1}\) and \(\bx_k^+\ne \bm{y}_{4}^{(k)}\)},\\
&\{\bx_k, \bx_{k}^{+}, \bm{y}_{1}^{(k)},\bm{y}_{2}^{(k)},\bm{y}_{3}^{(k)}, \bm{y}_{5}^{(k)}\}, \ \text{otherwise}.
\end{aligned}
\right. 
\]
Notice that the function values at all of the above interpolation points are already evaluated, and hence, such an interpolation does not cost more function evaluations. 


The quadratic models are important for the property of the iteration obtained by solving the 2-dimensional trust-region subproblems. To summarize the results above, we give the following remarks about how our method obtains the quadratic models. Table \ref{intcondtable} gives the interpolation conditions for all models.

\begin{table}[htbp]
\centering
 \fontsize{10}{11}\selectfont
\caption{Interpolation conditions for models used in \texttt{2D-MoSub}\label{intcondtable}}
\begin{tabular}{lll}
\toprule 
Model & Dim. & 
{Interpolation conditions}\\
\midrule 
$Q_k^{\rm sub}$ & 1 &
\(
Q_k^{\text{sub}}(\alpha)=Q_{k-1}^{+}(\alpha,0)
\) 
\\
\(Q_k\) & 2 
&
\(
Q_k(\T_{\dd_1^{(k)},\dd_2^{(k)}}^{(k)}(\z))=f(\z), \ \forall \ \z \in \{\y_1^{(k)},\y_2^{(k)},\y_3^{(k)}\} \ \&\  
Q_k(\alpha, 0)=Q_{k}^{\text{sub}}(\alpha)
\) \\
$Q_k^{+}$ & 2  & 
\(
Q_k^{+}(\T_{\dd_1^{(k+1)},\dd_*^{(k+1)}}^{(k+1)}(\bm{z}))=f(\bm{z}), \ \forall\ 
\z \in \mathcal{Y},\ \text{where \(\mathcal{Y}\subset \mathcal{Y}^+_k\) and \(\vert\mathcal{Y}\vert=6\)}
\) 
\\
$Q_k^{\text{mod}}$ & 2 
&
\(
Q_k^{\rm mod}(\T_{\dd_1^{(k)},\dd_2^{(k)}}^{(k)}(\bm{z}))=f(\bm{z}), \ \forall\ 
\z \in \mathcal{Y}_k^{\rm mod}  
\) 
\\
\bottomrule 
\end{tabular}
\end{table}

Considering that normally \(\dd_1^{(k)}=\frac{\bx_k-\bx_{k-1}}{\left\|\bx_k-\bx_{k-1}\right\|_2}\), and it is approximately a gradient descent direction, it can be seen that the new model \(Q_k\) inherits the good property of the models \(Q_{k-1}\) and \(Q_{k-1}^{+}\)
along the approximate gradient descent direction. Besides, it also uses the function value information at 3 new interpolation points in the \(k\)-th subspace, which at least forms a model with a fully linear property.

\subsection{Trust-region trial step}

\texttt{2D-MoSub} solves a 2-dimensional trust-region subproblem to find the optimal trial step within the trust region. It then evaluates the trial step's quality using the ratio of the function value improvement to the model value improvement. According to the ratio and pre-defined thresholds, the algorithm updates the subspace, interpolation set, trust-region parameters, and solution.

\begin{figure*}
\begin{center}
\begin{mdframed}[
    linewidth=0.8pt,
    hidealllines=true,
    topline=true,
    bottomline=true,
    frametitle={Step 3. (Trust-region trial step)}]
 \fontsize{8}{7}\selectfont
      \setstretch{0.1} 
\begin{algorithmic}[1]
  \State Solve the trust-region subproblem
\[
\begin{aligned}
\min_{\alpha,\beta} \ &Q_k(\alpha,\beta)\\
\text{\rm subject to} \ & \alpha^2+\beta^2\leq \Delta_k^2,
\end{aligned}
\]
and obtain \(\alpha^{(k)}\) and \(\beta^{(k)}\). Then let 
\[ 
\begin{aligned}
{\bx}^{\rm pre}_k&=\bx_k+\alpha^{(k)} \dd_1^{(k)}+\beta^{(k)} \dd_2^{(k)},\\
\bx^+_k&=\min_{\bx\in\{{\bx}_k, {\bx}^{\rm pre}_k,\y_1^{(k)},\y_2^{(k)},\y_3^{(k)}\}}\ f(\bx). 
\end{aligned}
\] 

\State If \(\bx_{k}^+\in \{\bx_k,\bx_{k-1}\}\), then let 
\(
\bx_{k+1}=\bx_k, 
\) 
\(\Delta_{k+1}=\Delta_k\) instead of (\ref{updatingDelta}), 
and 
\(
\dd_1^{(k+1)}=\dd_1^{(k)}
\) 
instead of (\ref{updatingofdd1}), 
and go to {\bf Step 4}.

\State  Otherwise, compute 
\[
\rho_k=\frac{f({\bx}^{+}_k)-f(\bx_{k})}{Q_k(\T_{\dd_1^{(k)},\dd_2^{(k)}}^{(k)}(\bx_k^+))-Q_k(0,0)}. 
\] 
\State If \(\rho_k \geq \eta\) or \({\bx}^{+}_k\in\{\y_1^{(k)}, \y_2^{(k)}, \y_3^{(k)}\}\), then let \({\bx}_{k+1}={\bx}^{+}_k\) and go to {\bf Step 4}. 

\State Otherwise, obtain the modified model \(Q_k^{\rm mod}\) by (\ref{intercondQkmod}) and solve the trust-region subproblem
\[
\begin{aligned}
\min_{\alpha,\beta} \quad &Q_k^{\rm mod}(\alpha,\beta)\\
\text{\rm subject to} \quad & \alpha^2+\beta^2\leq \Delta_k^2,
\end{aligned}
\]
and obtain \(\alpha^{(k,{\rm mod})}\) and \(\beta^{(k,{\rm mod})}\). Then 
\[
{\bx}^{\rm mod}_k=\bx_k+\alpha^{(k,{\rm mod})} \dd_1^{(k)}+\beta^{(k,{\rm mod})} \dd_2^{(k)}.
\]

\State If \(\bx_{k}^{\rm mod}\in \{\bx_k,\bx_{k-1}\}\), then let 
\(
\bx_{k+1}=\bx_k, 
\) 
\(\Delta_{k+1}=\Delta_k\) instead of (\ref{updatingDelta}), 
and 
\(
\dd_1^{(k+1)}=\dd_1^{(k)}
\) 
instead of (\ref{updatingofdd1}), 
and go to {\bf Step 4}. 

\State Otherwise, set
\[
\bx_{k}^{+}=\arg\min_{\bx\in \{\bx_{k}^{+},\bx_{k}^{\rm mod}\}}\ f(\bx). 
\]
Compute 
\[
\rho_k=\frac{f({\bx}^{+}_k)-f(\bx_{k})}{Q_k(\T_{\dd_1^{(k)},\dd_2^{(k)}}^{(k)}(\bx_k^+))-Q_k(\bx_{k})}. 
\]

\State  If \(\rho_k \geq \eta_0\), then \(\bx_{k+1}=\bx_k^+\) and go to {\bf Step 4}.

\State Otherwise, let 
\(
\bx_{k+1}=\bx_k, 
\)
and 
\(
\dd_1^{(k+1)}=\dd_1^{(k)}
\) 
instead of (\ref{updatingofdd1}), 
and go to {\bf Step 4}.

\end{algorithmic}
\end{mdframed}
\end{center}
\end{figure*}

In the trust-region trial step, we iteratively refine the solution by basically performing the following actions. First, we solve the trust-region subproblem to find the minimizer  of the quadratic model \(Q_k(\alpha, \beta)\) subject to the trust-region constraint. {Next, we calculate the predicted function value improvement based on \(Q_k\) and the actual function improvement based on \(f\). Then, we compute the ratio between the actual function improvement and the predicted function value improvement and compare it with a pre-defined threshold. Then we update \(\Delta_k\) and \(\bm{d}_1^{(k)}\) accordingly.} Step 3 includes more details, such as the use of \(Q_k^{\rm mod}\). The above statement only shows the basic idea of the trust-region trial step. Other details are in the algorithmic pseudo-code. 

By following these steps and making necessary adjustments based on the predicted and actual function-value improvements, we aim to iteratively refine the solution within the trust-region framework until the convergence is achieved. In the trust-region trial step, we perform as the algorithmic pseudo-code of Step 3 shows. The trust-region subproblem is solved by the truncated conjugated gradient method \cite{steihaug1983conjugate,toint1981towards} in the test implementation.

\subsection{Updating the trust-region radius and subspace}

Similar with the traditional trust-region methods, the trust-region radius of \texttt{2D-MoSub} is updated based on the trial step's quality. If the radius falls below a lower threshold, \texttt{2D-MoSub} terminates. Otherwise, the subspace is updated by calculating the new direction based on the updated solution. \texttt{2D-MoSub} constructs a new 1-dimensional quadratic interpolation model in the updated subspace. The details are shown in the algorithmic pseudo-code of the updating step numerated as Step 4 of \texttt{2D-MoSub}. 

\begin{figure*}
\begin{center}
\begin{mdframed}[
    linewidth=0.8pt,
    hidealllines=true,
    topline=true,
    bottomline=true,
    frametitle={Step 4. (Updating)}]
 \fontsize{8}{7}\selectfont
      \setstretch{0.1} 
\begin{algorithmic}[1]

\State If \(\Delta_k<\Delta_{\text{low}}\), then terminate. 

\State Otherwise,  update \(\Delta_{k+1}\) by 
\begin{equation}
\label{updatingDelta}
\Delta_{k+1} = 
\left\{
\begin{aligned}
\gamma_1 \Delta_k, & \text { if } \rho_k \geq \eta, \\ 
\gamma_2 \Delta_k, & \text { otherwise}.
 \end{aligned}
\right.
\end{equation}  
\State Let 
\begin{equation}
\label{updatingofdd1}
\dd_{1}^{(k+1)}=\frac{\bx_{k+1}-\bx_{k}}{\Vert \bx_{k+1}-\bx_{k} \Vert_2}.
\end{equation} 

\State Update \(Q_k\) to \(Q_k^{+}\) according to the interpolation conditions (\ref{Q+condconclude}). 

\State Obtain the 1-dimensional model \(Q_{k+1}^{\text{sub}}\) that satisfies (\ref{Qk+1subconditon}). 

\State Increment \(k\) by one and go to {\bf Step 1}.
\end{algorithmic}
\end{mdframed}
\end{center}
\end{figure*}

Fig. \ref{The_iterative_case_at_k} shows the iterative case at the \(k\)-th step of \texttt{2D-MoSub}.

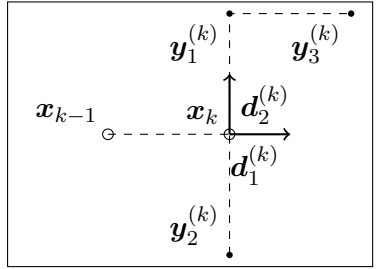
\begin{figure}[h]
\centering 
\fbox{
\begin{tikzpicture}[scale=0.8]

\coordinate (xk) at (2,0);
\coordinate (xk1) at (4,0);
\coordinate (y1) at (4,2);
\coordinate (y2) at (4,-2);
\coordinate (y3) at (6,2);

\draw (xk) circle [radius=2.5pt] node[above left] {$\bm{x}_{k-1}$};
\draw (xk1) circle [radius=2.5pt] node[above left] {$\bm{x}_{k}$};
\fill (y1) circle [radius=1.5pt] node[below left] {$\bm{y}_1^{(k)}$};
\fill (y2) circle [radius=1.5pt] node[above left] {$\bm{y}_2^{(k)}$};
\fill (y3) circle [radius=1.5pt] node[below left] {$\bm{y}_3^{(k)}$};

\draw[dashed] (xk) -- (xk1);
\draw[dashed] (y2) -- (y1) -- (y3);

\draw[->,thick] (xk1) -- (5,0) node[below left] {$\bm{d}_1^{(k)}$};
\draw[->,thick] (xk1) -- (4,1) node[below right] {$\bm{d}_2^{(k)}$};


\end{tikzpicture}
}
\caption{The iterative case at the \(k\)-th step and the subspace $\bm{x}_{k}+\text{span}\{\dd_1^{(k)},\dd_2^{(k)}\}$\label{The_iterative_case_at_k}}
\end{figure}

\section{Poisedness and quality of the interpolation set}
\label{section:Interpolation set's poisedness and interpolation quality}

{We give the discussion and analysis of the poisedness and quality of the interpolation set when constructing the quadratic interpolation model \(Q_k\) at each step in this section. As we have noticed, the interpolation model's quality in a region of interest is determined by the position of the interpolation points. For example, if a model \(Q\) interpolates the objective function \(f\) at points far away from a certain region of interest, the model value may differ greatly from the value of the objective function \(f\) in that region. The \(\Lambda\)-poisedness is a concept to measure how well a set of points is distributed, and ultimately how well an interpolation model will estimate the objective function.

The most commonly used metric for quantifying how well points are positioned in a region of interest is based on Lagrange polynomials. Given a set of \(p\) points \(\mathcal{Y}=\left\{\bm{y}_1, \dots, \bm{y}_p\right\}\), a basis of Lagrange polynomials satisfies 
\begin{equation*}
\ell_j(\bm{y}_i)= 
\left\{
\begin{aligned}
1, & \text{ if } i=j, \\ 
0, & \text{ otherwise}.
\end{aligned}
\right. 
\end{equation*}

In the following, we present the classic definition of \(\Lambda\)-poisedness \cite{conn2009introduction}.

\begin{definition}[\(\Lambda\)-poisedness] A set of points \(\mathcal{Y}\) is said to be \(\Lambda\)-poised on a set \(\mathcal{B}\) if \(\mathcal{Y}\) is linearly independent and the Lagrange polynomials \(\left\{\ell_1, \dots, \ell_p\right\}\) associated with \(\mathcal{Y}\) satisfy
\begin{equation*}
\Lambda \geq \max _{1 \leq i \leq p} \max _{\bm{x}\in\mathcal{B}}\left|\ell_i(\bm{x})\right| . 
\end{equation*}

\end{definition}

In our case, considering that the new 2-dimensional model \(Q_k\) formed by our algorithm has 3 fixed coefficients and 3 undetermined coefficients waiting to be determined by the interpolation conditions (\ref{interpolation-cond}), we give the following definitions and discussion. 

\begin{definition}[Basis function]
Given \(\y_1^{(k)},\y_2^{(k)},\y_3^{(k)},\bx_k\) and \(\dd_1^{(k)},\dd_2^{(k)}\in\Re^n\), let   
\begin{equation*}
\begin{aligned}
(\alpha_i^{(k)},\beta_i^{(k)})^{\top}&=\T_{\dd_1^{(k)},\dd_2^{(k)}}^{(k)}(\y_i^{(k)}),\ i=1,2,3, 
\end{aligned} 
\end{equation*} 
the basis matrix\footnote{We discuss the invertible case.} 
\begin{equation*}
\bm{W}= \setlength\arraycolsep{5pt}
\begin{bmatrix}
\beta_1^{(k)} & (\beta_1^{(k)})^2 & \alpha_1^{(k)}\beta_1^{(k)} \\
\beta_2^{(k)} & (\beta_2^{(k)})^2 & \alpha_2^{(k)}\beta_2^{(k)}  \\
\beta_3^{(k)} & (\beta_3^{(k)})^2 & \alpha_3^{(k)}\beta_3^{(k)} 
\end{bmatrix}, 
\end{equation*} 
\(
c_0^{(i)}=f(\bx_k)+a^{(k)}\alpha_i^{(k)}+b^{(k)}(\alpha_i^{(k)})^2  
\) for \(i=1,2,3\), 
and 
\begin{equation*}
\begin{bmatrix}
c_{1} \\
d_{1} \\
e_{1}
\end{bmatrix}=\W^{-1}\begin{bmatrix}
1-c_0^{(1)} \\
-c_0^{(1)} \\
-c_0^{(1)}
\end{bmatrix},\ \begin{bmatrix}
c_{2} \\
d_{2} \\
e_{2}
\end{bmatrix}=\W^{-1}\begin{bmatrix}
-c_0^{(2)} \\
1-c_0^{(2)} \\
-c_0^{(2)}
\end{bmatrix},\ 
\begin{bmatrix}
c_{3} \\
d_{3} \\
e_{3}
\end{bmatrix}=\W^{-1}\begin{bmatrix}
-c_0^{(3)} \\
-c_0^{(3)} \\
1-c_0^{(3)}
\end{bmatrix}. 
\end{equation*}
The basis functions are 
\begin{equation}
\label{basisfunction}
\ell_i(\alpha, \beta)=c_0^{(i)}+ c_i\beta+ d_i\beta^2+ e_i\alpha \beta, \ i=1,2,3.  
\end{equation}
\end{definition}

We now define the \(\Lambda\)-poisedness in our case. Notice that \(\Delta_k>0\).

\begin{definition}[\(\Lambda\)-poisedness for 2-dimensional case with 3 determined coefficients]
\label{newly-defined_poised} 
A set of 3 points \(\mathcal{Y}\subset \Re^2\) is said to be \(\Lambda\)-poised on a set \(\{(\alpha,\beta)^{\top}, \left\Vert (\alpha,\beta)\right\Vert_{\infty} \leq \Delta_k\}\) if \(\mathcal{Y}\) is linearly independent and the basis polynomials \(\left\{\ell_1, \ell_2,\ell_3\right\}\) associated with \(\mathcal{Y}\) in (\ref{basisfunction}) satisfy
\begin{equation}
\Lambda \geq \max_{1 \leq i \leq 3} \max_{\left\Vert (\alpha,\beta)\right\Vert_{\infty} \leq \Delta_k} \ \vert \ell_i(\alpha, \beta)\vert.
\end{equation}

\end{definition}

\begin{remark}\label{most_poised}
The most poised interpolation set is 1-poised. The region in Definition \ref{newly-defined_poised} is an \(\ell_{\infty}\)-ball for obtaining the analytic solution in the following without loss of generality of describing a region. 
\end{remark}

}

\begin{theorem}
\label{theoremofbasisfunctions}
At the step of constructing the quadratic interpolation model at each iteration, \texttt{2D-MoSub} has the following Lagrange basis function for computing. 

In the case where \(f(\bx_{k}) \le f(\y_1^{(k)})\), it holds that  
\begin{equation}
\label{ell1conclusion}
\left\{
\begin{aligned}
&\ell_1(\alpha, \beta)=c_0^{(1)}+\frac{1}{2 \Delta_k} \beta+\frac{1-2 c_0^{(1)}}{2 \Delta_k^2} \beta^2+\left(-\frac{1}{\Delta_k^2}\right) \alpha \beta,\\
&\ell_2(\alpha, \beta)=c_0^{(2)}+\left(-\frac{1}{2 \Delta_k}\right) \beta+\frac{1-2 c_0^{(2)}}{2 \Delta_k^2} \beta^2,\\
&\ell_3(\alpha, \beta)=c_0^{(3)}+\left(-\frac{c_0^{(3)}}{\Delta_k^2}\right) \beta^2+\frac{1}{\Delta_k^2} \alpha \beta, 
\end{aligned}
\right. 
\end{equation}
and in the case where \(f(\bx_{k}) > f(\y_1^{(k)})\), it holds that
\begin{equation}
\label{ell2conclusion}
\left\{
\begin{aligned}
&\ell_1(\alpha, \beta)=c_0^{(1)}+\frac{4-3 c_0^{(1)}}{2 \Delta_k} \beta+\frac{-2+c_0^{(1)}}{2 \Delta_k^2} \beta^2+\left(-\frac{1}{\Delta_k^2}\right) \alpha \beta,\\
&\ell_2(\alpha, \beta)=c_0^{(2)}+\left(-\frac{1+3 c_0^{(2)}}{2 \Delta_k}\right) \beta+\frac{1+c_0^{(2)}}{2 \Delta_k^2} \beta^2,\\
&\ell_3(\alpha, \beta)=c_0^{(3)}+\left(-\frac{3 c_0^{(3)}}{2 \Delta_k}\right) \beta+\frac{c_0^{(3)}}{2 \Delta_k^2} \beta^2+\frac{1}{\Delta_k^2} \alpha \beta,
\end{aligned}
\right.
\end{equation}
where \(c_0^{(1)}=f(\bx_k),\ c_0^{(2)}=f(\bx_k),\ c_0^{(3)}=f(\bx_k)+a^{(k)}\Delta_k+b^{(k)}\Delta_k^2\). 
\end{theorem}

\begin{proof} 
In the case where \(f(\bx_k)\le f(\y_1^{(k)})\), it holds that 
\begin{equation}
\W_1=\begin{bmatrix}
\Delta_k & \Delta_k^2 & 0 \\
-\Delta_k & \Delta_k^2 & 0 \\
\Delta_k & \Delta_k^2 & \Delta_k^2
\end{bmatrix},
\end{equation}
and 
\begin{equation*}
\begin{bmatrix}
c_{1} \\
d_{1} \\
e_{1}
\end{bmatrix}=\W_1^{-1}\begin{bmatrix}
1-c_0^{(1)} \\
-c_0^{(1)} \\
-c_0^{(1)}
\end{bmatrix},\ \begin{bmatrix}
c_{2} \\
d_{2} \\
e_{2}
\end{bmatrix}=\W_1^{-1}\begin{bmatrix}
-c_0^{(2)} \\
1-c_0^{(2)} \\
-c_0^{(2)}
\end{bmatrix},\ \begin{bmatrix}
c_{3} \\
d_{3} \\
e_{3}
\end{bmatrix}=\W_1^{-1}\begin{bmatrix}
-c_0^{(3)} \\
-c_0^{(3)} \\
1-c_0^{(3)}
\end{bmatrix}. 
\end{equation*} 
In the case where \(f(\bx_k) > f(\y_1^{(k)})\), it holds that  
\begin{equation}
\W_2=\begin{bmatrix}
\Delta_k & \Delta_k^2 & 0 \\
2 \Delta_k & 4 \Delta_k^2 & 0 \\
\Delta_k & \Delta_k^2 & \Delta_k^2
\end{bmatrix}, 
\end{equation}
and
\begin{equation*}
\begin{bmatrix}
c_{1} \\
d_{1} \\
e_{1}
\end{bmatrix}=\W_2^{-1}\begin{bmatrix}
1-c_0^{(1)} \\
-c_0^{(1)} \\
-c_0^{(1)}
\end{bmatrix},\ \begin{bmatrix}
c_{2} \\
d_{2} \\
e_{2}
\end{bmatrix}=\W_2^{-1}\begin{bmatrix}
-c_0^{(2)} \\
1-c_0^{(2)} \\
-c_0^{(2)}
\end{bmatrix},\ \begin{bmatrix}
c_{3} \\
d_{3} \\
e_{3}
\end{bmatrix}=\W_2^{-1}\begin{bmatrix}
-c_0^{(3)} \\
-c_0^{(3)} \\
1-c_0^{(3)}
\end{bmatrix}. 
\end{equation*}

Therefore, (\ref{ell1conclusion}) and (\ref{ell2conclusion}) hold.  
\end{proof}

Fig. \ref{fourDifferent casesfory1y2y3} shows the different cases for \(\y_1^{(k)},\y_2^{(k)},\y_3^{(k)}\).

\begin{figure}[htbp]
\centering 
\fbox{
\begin{tikzpicture}[scale=0.5]

\coordinate (y1) at (4,2);
\coordinate (y2) at (4,-2);
 \coordinate (y2fake) at (4,2.7);
\coordinate (y3) at (6,2);

\fill (y1) circle [radius=2.5pt] node[below left] {$\bm{y}_1^{(k)}$};
\fill (y2) circle [radius=2.5pt] node[above left] {$\bm{y}_2^{(k)}$};
\fill (y2fake) circle [radius=0pt];
\fill (y3) circle [radius=2.5pt] node[below right] {$\bm{y}_3^{(k)}$};


\draw[dashed] (y2) -- (y1) node[midway, right] {\(2\Delta_k\)} -- (y3) node[midway, below] {\(\Delta_k\)};

\end{tikzpicture}
}
\ 
\fbox{
\begin{tikzpicture}[scale=0.5]

\coordinate (y1) at (4,2);
\coordinate (y2) at (4,-2);
 \coordinate (y2fake) at (4,2.7);
\coordinate (y3) at (6,-2);

\fill (y1) circle [radius=2.5pt] node[below left] {$\bm{y}_1^{(k)}$};
\fill (y2) circle [radius=2.5pt] node[above left] {$\bm{y}_2^{(k)}$};
\fill (y2fake) circle [radius=0pt];
\fill (y3) circle [radius=2.5pt] node[above right] {$\bm{y}_3^{(k)}$};


\draw[dashed] (y1) -- (y2) node[midway, right] {\(2\Delta_k\)} -- (y3) node[midway, above] {\(\Delta_k\)};

\end{tikzpicture}
}\  \fbox{
\begin{tikzpicture}[scale=0.41]


\coordinate (y1) at (4,2);
\coordinate (y2fake) at (4,1);
\coordinate (y2) at (4,4);
\coordinate (y3) at (6,2);

\fill (y1) circle [radius=2.5pt] node[below left] {$\bm{y}_1^{(k)}$};
\fill (y2) circle [radius=2.5pt] node[above left] {$\bm{y}_2^{(k)}$};
\fill (y2fake) circle [radius=0pt];
\fill (y3) circle [radius=2.5pt] node[below right] {$\bm{y}_3^{(k)}$};


\draw[dashed] (y2) -- (y1) node[midway, right] {\(\Delta_k\)} -- (y3) node[midway, below] {\(\Delta_k\)};


\end{tikzpicture}
}\   \fbox{
\begin{tikzpicture}[scale=0.41]


\coordinate (y1) at (4,2);
\coordinate (y2fake) at (4,1);
\coordinate (y2) at (4,4);
\coordinate (y3) at (6,4);

\fill (y1) circle [radius=2.5pt] node[below left] {$\bm{y}_1^{(k)}$};
\fill (y2) circle [radius=2.5pt] node[above left] {$\bm{y}_2^{(k)}$};
\fill (y2fake) circle [radius=0pt];
\fill (y3) circle [radius=2.5pt] node[below right] {$\bm{y}_3^{(k)}$};


\draw[dashed] (y1) -- (y2) node[midway, right] {\(\Delta_k\)} -- (y3) node[midway, above] {\(\Delta_k\)};


\end{tikzpicture}
}
\caption{Different cases for \(\y_1^{(k)},\y_2^{(k)},\y_3^{(k)}\)\label{fourDifferent casesfory1y2y3}}
\end{figure}
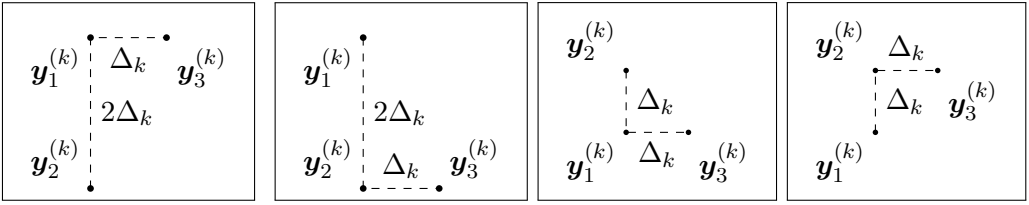

\begin{proposition}
\begin{sloppypar}
In the case where \(f(\bx_k)\le f(\y_1^{(k)})\), it holds that the interpolation set \(\{\y_1^{(k)},\y_2^{(k)},\y_3^{(k)}\}\) is \(\Lambda_1\)-poised, where
\( 
\Lambda_1=2.
\) 
In the case where \(f(\bx_k)> f(\y_1^{(k)})\), it holds that the interpolation set \(\{\y_1^{(k)},\y_2^{(k)},\y_3^{(k)}\}\) is \(\Lambda_2\)-poised, where
\(
\Lambda_2\le \max\{4, 1+3\Delta_k (\vert a^{(k)} \vert+ \vert b^{(k)}\vert\Delta_k) \}.
\) 
\end{sloppypar}
\end{proposition}

\begin{proof}
The conclusion holds straightfoward according to the analytic solutions of the 2-dimensional problem
\begin{equation*}
\max _{1 \leq i \leq 3} \max_{\left\Vert (\alpha,\beta)\right\Vert_{\infty} \leq \Delta_k} \ \vert \ell_i(\alpha, \beta)\vert 
\end{equation*}
for \(\ell_1,\ell_2,\ell_3\) in the two cases in Theorem \ref{theoremofbasisfunctions} and this proposition. 
\end{proof}

Considering Remark \ref{most_poised}, the interpolation sets above used by \texttt{2D-MoSub} are sufficiently poised on a 2-dimensional subspace.

\section{Some properties of \texttt{2D-MoSub}}

\label{section:Some properties of 2D-MoSub}

The main idea of our new proposed subspace derivative-free optimization method \texttt{2D-MoSub} is to iteratively obtain the iteration point by minimizing a quadratic model function within the trust region in a 2-dimensional subspace. The quadratic model in the current 2-dimensional subspace and the one defined in one dimension of the 2-dimensional subspace inherit good properties of the previous subspace, model, and iteration point. 

To construct a determined quadratic interpolation model function at each step, \texttt{2D-MoSub} adds 3 new interpolation points to take up the other 3 coefficients of the 2-dimensional quadratic model after 3 coefficients have been already taken up by the previous model.

The following discusses some properties of our method. Notice that (theoretically) for \(\bx_{k+1}=\bx_k+\alpha^{(k)}\dd_1^{(k)}+\beta^{(k)}\dd_2^{(k)}\), \(\alpha^{(k)}\) and \(\beta^{(k)}\) satisfy that 
\[ 
(\alpha^{(k)},\beta^{(k)})^{\top} \in  \{\arg \min_{\alpha,\beta} \ Q_k(\alpha,\beta), \text { subject to } \alpha^2+\beta^2 \leq \Delta_k^2\}. 
\]

We have the following proposition. 

\begin{proposition} 
It holds that 
\begin{equation}
\label{equ123}
\left\{
\begin{aligned}
&\min_{\alpha^2+\beta^2 \leq \Delta_k^2} \ Q^{+}_k(\T_{\dd_1^{(k+1)},\dd_*^{(k+1)}}^{(k+1)}(\bx_k+\alpha \dd_1^{(k)}+\beta \dd_2^{(k)})) \leq f(\bx_{k+1}), \\
&\min_{-\Delta_k \le \alpha \leq \Delta_k} \ Q^{\rm sub}_{k+1}(\hat{\T}_{\dd_1^{(k+1)}}^{(k+1)}(\bx_k+\alpha \dd_1^{(k+1)}))\leq f(\bx_{k+1}),\\
&\min_{\alpha^2+\beta^2 \leq \Delta^2_{k+1}} \ Q_{k+1}(\T_{\dd_1^{(k+1)},\dd_2^{(k+1)}}^{(k+1)}(\bx_{k+1}+\alpha \dd_1^{(k+1)}+\beta \dd_2^{(k+1)}))   \leq f(\bx_{k+1}). 
\end{aligned}
\right. 
\end{equation}
\end{proposition}
\begin{proof}
The proof of the proposition is direct based on the definitions. 
\end{proof}

In conclusion, the way of updating our new model has two advantages. One is that the model \(Q_{k+1}\) sufficiently considers the property of \(Q_k\) along the 1-dimensional subspace \(\bx_{k+1}+\text{span}\left\{\dd_1^{(k+1)}\right\}\), since \(\bx_{k+1}\) itself is already a successful step given by \(Q_k\). The other advantage or necessary property is that minimizing \(Q_{k+1}\) over a trust region can give an iteration point with a non-increasing model function value according to (\ref{equ123}).

{In addition, the quadratic model $Q_k$ obtained by \texttt{2D-MoSub} is exactly a solution of the following subproblem, and it denotes that it is the same as \(Q_{k-1}^{+}\) along the direction \(\dd_1^{(k)}\), which is an approximate gradient descent direction numerically.  
\begin{subproblem}[Least \(L^2\)-norm along a direction updating]
\begin{equation}
\label{subproblemofqk}
\begin{aligned}
\min_{Q} & \quad \int_{-\infty}^{\infty}\left( Q(\alpha, 0)-Q_{k-1}^{+}(\alpha, 0)\right)^2 d\alpha \\
\text {\rm subject to} &\quad Q(\T_{\dd_1^{(k)}, \dd_2^{(k)}}^{(k)}(\z))= f(\z), \ \forall\ \z \in \{\bx_k, \y_1^{(k)}, \y_2^{(k)}, \y_3^{(k)}\}. 
\end{aligned}
\end{equation}
\end{subproblem} 


We will show the convexity of subproblem (\ref{subproblemofqk}). 

\begin{theorem}

The subproblem (\ref{subproblemofqk}) is strictly convex.

\end{theorem}

\begin{proof}
What we need to show is that the objective function is strictly convex as a function of \(Q\), i.e., for \(0<c<1\) and 2-dimensional quadratic functions \(Q_a\) and \(Q_b\), it holds that 
\begin{equation}
\label{wantandneedtoprove}
\begin{aligned}
&\int_{-\infty}^{\infty}\left( c Q_a(\alpha, 0)+(1-c) Q_b(\alpha, 0)-Q_{k-1}^{+}(\alpha, 0)\right)^2 d\alpha\\
<&c\int_{-\infty}^{\infty}\left(Q_a(\alpha, 0)-Q_{k-1}^{+}(\alpha, 0)\right)^2d\alpha+(1-c)\int_{-\infty}^{\infty}\left(Q_b(\alpha, 0)-Q_{k-1}^{+}(\alpha, 0)\right)^2d\alpha.
\end{aligned}
\end{equation}

It holds that the difference between the right-hand side of (\ref{wantandneedtoprove}) and the left-hand side of (\ref{wantandneedtoprove}) is 
\begin{equation*}
\begin{aligned}
&-2 c(1-c) \int_{-\infty}^{\infty}Q_a(\alpha, 0) Q_b(\alpha, 0) d\alpha+(c-c^2) \int_{-\infty}^{\infty}(Q_a(\alpha, 0))^2d\alpha \\
&+(1-c-(1-c)^2) \int_{-\infty}^{\infty}(Q_b(\alpha, 0))^2d\alpha \\
= & (c-c^2) \int_{-\infty}^{\infty}(Q_a(\alpha, 0)-Q_b(\alpha, 0))^2d\alpha<0,
\end{aligned}
\end{equation*} 
since  \(0<c<1\). Therefore, we obtain the strictly convex of the objective function of the subproblem. 
\end{proof}

The theorem above shows that the model \(Q_{k}\) obtained by \texttt{2D-MoSub} is exactly the unique solution of the subproblem (\ref{subproblemofqk}). The above shows the method \texttt{2D-MoSub}'s advantage when \(\dd_1^{(k)}\) is an approximate gradient descent direction. 

\begin{theorem}
If \(Q_k\) is the solution of subproblem (\ref{subproblemofqk}), then for quadratic function \(f\), it holds that 
\begin{equation}
\label{proj-1} 
\begin{aligned}
\int_{-\infty}^{\infty}\left(Q_k(\alpha, 0)-\tilde{f}(\alpha, 0)\right)^2d\alpha=&\int_{-\infty}^{\infty}\left(Q_{k-1}^{+}(\alpha, 0)-\tilde{f}(\alpha, 0)\right)^2d\alpha\\
&-\int_{-\infty}^{\infty}\left(Q_k(\alpha, 0)-Q_{k-1}^{+}(\alpha, 0)\right)^2d\alpha,
\end{aligned}
\end{equation}
where \(\tilde{f}=f\circ \T_{\dd_1^{(k)},\dd_2^{(k)}}^{(k)}\). 
\end{theorem}
 
\begin{proof}
Let \(Q_t=Q_k+t\left(Q_k-\tilde{f}\right)\), where \(t \in {\Re}\). Then \(Q_t\) is a quadratic function and it satisfies the interpolation conditions of subproblem (\ref{subproblemofqk}). Based on the optimality of \(Q_k\), we can know that the quadratic function 
\begin{equation*}
\varphi(t):=\int_{-\infty}^{\infty}\left(Q_t(\alpha, 0)-Q_{k-1}^{+}(\alpha, 0)\right)^2 d\alpha 
\end{equation*} 
achieves its minimum function value when \(t=0\). We expand \(\varphi(t)\), and then obtain that 
\begin{equation*}
\begin{aligned}
\varphi(t)= & t^2\int_{-\infty}^{\infty}\left(Q_k(\alpha, 0)-\tilde{f}(\alpha, 0)\right)^2d\alpha\\
&+2 t\int_{-\infty}^{\infty}\left(Q_k(\alpha, 0)-\tilde{f}(\alpha, 0)\right)\left(Q_k(\alpha, 0)-Q_{k-1}^{+}(\alpha, 0)\right)d\alpha\\
&+\int_{-\infty}^{\infty}\left(Q_k(\alpha, 0)-Q_{k-1}^{+}(\alpha, 0)\right)^2d\alpha, 
\end{aligned}
\end{equation*} 
and hence
\begin{equation*}
\int_{-\infty}^{\infty}\left(Q_k(\alpha, 0)-\tilde{f}(\alpha, 0)\right)\left(Q_k(\alpha, 0)-Q_{k-1}^{+}(\alpha, 0)\right)d\alpha=0. 
\end{equation*} 
Considering \(\varphi(-1)\), we obtain the conclusion. 
\end{proof}

In addition, the following corollary holds consequently. 

\begin{corollary}
If \(Q_k\) is the solution of subproblem (\ref{subproblemofqk}), then for quadratic function \(f\), it holds that 
\begin{equation}
\label{proj-2}
\begin{aligned}
\int_{-\infty}^{\infty}\left(Q_k(\alpha, 0)-\tilde{f}(\alpha, 0)\right)^2d\alpha\le \int_{-\infty}^{\infty}\left(Q_{k-1}^{+}(\alpha, 0)-\tilde{f}(\alpha, 0)\right)^2d\alpha, 
\end{aligned}
\end{equation}
where \(\tilde{f}=f\circ \T_{\dd_1^{(k)},\dd_2^{(k)}}^{(k)}\). 
\end{corollary}

\begin{proof}
The inequality (\ref{proj-2}) holds according to the equality (\ref{proj-1}) and the fact
\[
\int_{-\infty}^{\infty}\left(Q_k(\alpha, 0)-Q_{k-1}^{+}(\alpha, 0)\right)^2d\alpha\ge 0. 
\] 
\end{proof}

The above theorem and corollary about the projection property show the result that the model
\(Q_k\) given by our method \texttt{2D-MoSub} has a better approximation along the direction \(\dd_1^{(k)}\). 

}

In the following, we first present the decrement of the function value of \texttt{2D-MoSub}.

\begin{proposition}
The function values of the iteration points given by the method follows the decreasing relationship, which is
\[
f(\bx_{k+1})\le f(\bx_k), 
\]
and it holds that 
\[
f(\bx_{k+1})\le f(\bx_k)-\eta (Q_k(0,0)-Q_k(\T^{(k)}_{\dd_1^{(k)},\dd_2^{(k)}}(\bx_{k+1})))
\]
at the successful step.

\end{proposition}

\begin{proof}
The proof is straightforward according to the algorithmic framework and the criterion related to (\ref{ratioofsuccessfulstep}). 
\end{proof}

The following results reveal the convergence of \texttt{2D-MoSub}. 
\begin{assumption}
\label{assumption-convergence-subspace} 
The objective function $f$ is bounded from below, twice continuously differentiable, and its second-order derivative is bounded. Let $C$ represent an upper bound on $\left\Vert \nabla^2 f\right\Vert$. 
There exists an infinite set $\mathcal{K}\subseteq\mathbb{N}^+$ such that 

\noindent 
1) there exists \(\varepsilon_1 > 0\)  such that   
\( 
\left\|\bm{P}_k\nabla f(\bx_k)\right\|_2 \geq \varepsilon_1 \left\|\nabla f(\bx_k)\right\|_2 
\) for \(k \in \mathcal{K}\), 
where \(\bm{P}_k\) is the orthogonal projection from \({\Re}^n\) to the 2-dimensional subspace \(\mathcal{S}_{\dd_1^{(k)},\dd_2^{(k)}}^{(k)}\); 

\noindent 
2) $f(\bx_{k+1})-\inf_{\dd \in \mathcal{S}_{\dd_1^{(k)},\dd_2^{(k)}}^{(k)}} f(\bx_k+\dd) \rightarrow 0$ as $k \in \mathcal{K}$ and $k \rightarrow \infty$.
\end{assumption}

The following theorem follows and extends Theorem 5.7 in Zhang's thesis \cite{zhang012} {\color{black}(with details in its Section 5)}.

\begin{theorem}
If Assumption \ref{assumption-convergence-subspace} holds for the objective function \(f\) and \texttt{2D-MoSub}, then 
\begin{equation*}
\liminf_{k \rightarrow \infty}\left\|\nabla f(\bx_k)\right\|_2=0,
\end{equation*}
where each \(\bx_k\) is an iteration point generated by \texttt{2D-MoSub}. 
\end{theorem}

\begin{proof} 
The proof presented here follows and extends the proof of Theorem 5.7 in Zhang's thesis \cite{zhang012}. We want to prove that 
\(
\left\|\nabla f(\bx_k)\right\|_2 \rightarrow 0\), as \(k \in \mathcal{K}\) and \(k \rightarrow \infty\). We prove it by contradiction. Suppose it does not hold, and there exists \(\varepsilon_2 > 0\) and an infinite subset \(\overline{\mathcal{K}}\) of \(\mathcal{K}\) such that 
\(\left\|\nabla f(\bx_k)\right\|_2 \geq \varepsilon_2\), for \(k \in \overline{\mathcal{K}}\). According to Assumption \ref{assumption-convergence-subspace}, w.l.o.g., it can be assumed that for \(k \in \overline{\mathcal{K}}\), 
\begin{equation*}
\left\|\bm{P}_k\nabla f(\bx_k)\right\|_2 \geq \varepsilon_1 \left\|\nabla f(\bx_k)\right\|_2, 
\end{equation*} 
and 
\begin{equation}
\label{solutionquality}
f(\bx_{k+1})-\inf _{x \in \mathcal{S}_{\dd_1^{(k)},\dd_2^{(k)}}^{(k)}} f(\bx_k+\dd) \leq \frac{\varepsilon_1^2\varepsilon_2^2}{4 C}. 
\end{equation} 
For any \(k \in \overline{\mathcal{K}}\), it then holds that 
\begin{equation*}
\begin{aligned}
 &f(\bx_k)-f(\bx_{k+1})\\ 
= & {\left(f(\bx_k)-\inf _{\dd \in \mathcal{S}_{\dd_1^{(k)},\dd_2^{(k)}}^{(k)}} f(\bx_k+\dd)\right)-\left(f(\bx_{k+1})-\inf_{\dd \in \mathcal{S}_{\dd_1^{(k)},\dd_2^{(k)}}^{(k)}} f(\bx_k+\dd)\right)} \\
\geq & \frac{1}{2 C}\left\|\bm{P}_k \nabla f(\bx_k)\right\|_2^2-\frac{\varepsilon_1^2\varepsilon_2^2}{4 C}\\  
\geq & \frac{\varepsilon_1^2\varepsilon_2^2}{2 C}-\frac{\varepsilon_1^2\varepsilon_2^2}{4 C} \\
= & \frac{\varepsilon_1^2\varepsilon_2^2}{4 C},
\end{aligned}
\end{equation*} 
where the inequality in the third row follows (\ref{solutionquality}) and Lemma 5.5 in Zhang's thesis \cite{zhang012}. This contradicts the facts that \(\overline{\mathcal{K}}\) is an infinite subset and \(f\) is lower bounded. 
\end{proof}

Notice that Assumption \ref{assumption-convergence-subspace} is a sufficient condition for the convergence of our \texttt{2D-MoSub} method. We present the result of Lemma 5.5 in Zhang's thesis \cite{zhang012} as a remark.

\begin{remark}
Suppose that the objective function $f$ is bounded from below, twice continuously differentiable, and its second-order derivative is bounded. Let $C$ represent an upper bound on $\left\Vert \nabla^2 f\right\Vert$, and \(\mathcal{S}\) is a subspace of \(\Re^n\). Then it holds that  
\begin{equation*}
f(\bx)-\inf _{\bm{d} \in \mathcal{S}} f(\bx+\bm{d}) \geq \frac{1}{2 C}\|\bm{P} \nabla f(\bx)\|_2^2,
\end{equation*}
where $\bm{P}$ is the orthogonal projection from $\Re^n$ to $\mathcal{S}$. 
\end{remark}

\section{Numerical results}

\label{section:Numerical results}

We provide experimental results demonstrating the algorithm's performance on various optimization problems. The results include comparisons with existing methods, showcasing the efficiency and effectiveness of the proposed approach. Table \ref{Parameter settings} shows the parameter settings when testing our method. 

\begin{table}[htbp]
\centering
\caption{Parameter settings of \texttt{2D-MoSub}\label{Parameter settings}}
\begin{tabular}{ccl}
\toprule
{Parameter} & {Value} & {Description} \\
\midrule 
$\Delta_1$ & $1$ & Initial trust-region radius \\
$\Delta_{\text{low}}$ & $1 \times 10^{-4}$ & Lower bound on trust-region radius \\
$\Delta_{\text{upper}}$ & $1 \times 10^4$ & Upper bound on trust-region radius \\
$\gamma_1$ & $10$ & Increase factor for trust-region radius \\
$\gamma_2$ & $0.1$ & Decrease factor for trust-region radius \\
$\eta$ & $0.2$ & Threshold for successful step \\
 $\eta_0$ & $0.1$ & Threshold for {modified} successful step \\
\(\dd^{(1)}\) & \((1,0,\dots,0)\) & An initial direction \\ 
\bottomrule
\end{tabular}
\end{table}

To show the general numerical behavior of our subspace method, we try to solve the classic test problems and present the numerical results using the criterions called performance profile \cite{dolan2002benchmarking,audet2017derivative} and data profile \cite{BenchmarkingDFO,audet2017derivative}. The test problems with results in Fig. \ref{fig-perf} and Fig. \ref{fig-data} are shown in Table \ref{table5}, and they are from classic and common unconstrained optimization test functions collections including CUTEr/CUTEst \cite{CUTEr,gould2015cutest}.

\begin{table}[htbp]\ttfamily
  \centering   
    \caption{Test problems\label{table5}} 
      \setlength\tabcolsep{1.1pt}
    \fontsize{10}{10}\selectfont
  \begin{tabular}{llllllll}  
    \toprule  
    ARGLINA & ARGLINA4 & ARGLINB & ARGLINC & ARGTRIG & ARWHEAD &  
BDQRTIC &BDQRTICP \\
BDVALUE & BROWNAL & BROYDN3D & BROYDN7D &
BRYBND & CHAINWOO & CHEBQUAD & CHNROSNB \\
 CHPOWELLB & CHPOWELLS &
 CHROSEN & COSINE &
 CRAGGLVY &
 CUBE & CURLY10 & CURLY20 \\
CURLY30 & DIXMAANE &
 DIXMAANF & DIXMAANG &
 DIXMAANH & DIXMAANI &
DIXMAANJ & DIXMAANK \\
DIXMAANL & DIXMAANM & DIXMAANN &
 DIXMAANO &
  DIXMAANP & DQRTIC& 
  EDENSCH & ENGVAL1 \\
ERRINROS & EXPSUM &
EXTROSNB & EXTTET &
 FIROSE & FLETCBV2 & FLETCBV3 & FLETCHCR \\
FREUROTH & GENBROWN &
 GENHUMPS & GENROSE & INDEF & INTEGREQ &
  LIARWHD & LILIFUN3 \\
LILIFUN4 & MOREBV & MOREBVL & NCB20 &
 NCB20B & NONCVXU2 &
 NONCVXUN &
 NONDIA \\
NONDQUAR & PENALTY1 &
 PENALTY2 & PENALTY3 &
 PENALTY3P & POWELLSG &
 POWER & ROSENBROCK \\
SBRYBND & SBRYBNDL&
  SCHMVETT & SCOSINE & SCOSINEL &
 SEROSE &
SINQUAD & SPARSINE \\
SPARSQUR & SPMSRTLS & SROSENBR & STMOD &
 TOINTGSS & TOINTTRIG &
  TQUARTIC & TRIGSABS \\
TRIGSSQS & TRIROSE1 &
 TRIROSE2 & VARDIM & WOODS& - & - & - \\ 
    \bottomrule   
  \end{tabular}
\end{table}

The performance profile and data profile describe the number of function evaluations taken by the algorithm in the algorithm set \(\mathcal{A}\) to achieve a given accuracy when solving problems in a given problem set.

We define the value 
\[
f_{\mathrm{acc}}^{N}=\frac{f(\bx_{N})-f(\bx_{\text{int}})}{f(\bx_{\text{best}})-f(\bx_{\text{int}})} \in [0,1],
\] 
and the tolerance \(\tau \in [0,1]\), where \(\bx_{N}\) denotes the best point found by the algorithm after \(N\) function evaluations, \(\bx_{\text{int}}\) denotes the initial point, and \(\bx_{\text{best}}\) denotes the best known solution. When \(f_{\mathrm{acc}}^{N} \ge 1-\tau\),  we say that the solution reaches the accuracy \(\tau\). We give \(N_{s,p}=\min\{n \in \mathbb{N},\ f_{\mathrm{acc}}^{n}\ge 1-\tau \}\) and the definitions that
\begin{equation*}
\begin{aligned}
&T_{s, p}=\left\{\begin{aligned}
& 1,  \text{ if} \ f_{\mathrm{acc}}^{N} \geq 1-\tau \ \text{for some } N,\\
& 0,  \text{ otherwise},
\end{aligned}\right.
\end{aligned}
\end{equation*}
and
\begin{equation*}
\begin{aligned}
&r_{s, p}=\left\{\begin{aligned}
&\frac{N_{s, p}}{\min \left\{N_{\tilde{s}, p}: \tilde{s} \in \mathcal{A},\ T_{\tilde{s}, p}=1\right\}}, \text{ if} \ T_{s, p}=1, \\
&+\infty, \text{ otherwise},
\end{aligned}\right.
\end{aligned}
\end{equation*}
where \(s\) is the given solver or algorithm. For the given tolerance \(\tau\) and a certain problem \(p\) in the problem set \(\mathcal{P}\), the parameter \(r_{s, p}\) shows the ratio of the number of the function evaluations using the solver \(s\) divided by that using the fastest algorithm on the problem \(p\). 

In the performance profile, 
\[
\pi_{s}(\alpha)=\frac{1}{\vert\mathcal{P}\vert}\left\vert\left\{p \in \mathcal{P}: r_{s, p} \leq \alpha\right\}\right\vert, 
\] 
where \(\alpha\) refers to \(\frac{\rm NF}{{\rm NF}_{\text{min}}} \in [1, +\infty)\) in Fig. \ref{fig-perf}, and \(\vert\cdot\vert\) denotes the cardinality.

Notice that NF denotes the corresponding number of function evaluations solving test problems with an algorithm, and ${\rm NF}_{\text{min}}$ denotes the minimum number of function evaluations among all of the algorithms.

In the data profile, 
\[ 
\delta_s({\beta})=\frac{1}{\vert \mathcal{P}\vert}\left\vert \left\{p \in \mathcal{P}: N_{s, p} \leq {\beta}\left(n+1\right) T_{s, p}\right\}\right\vert. 
\] 

Notice that a higher value of \(\pi_s(\alpha)\) or \(\delta_s({\beta})\) represents solving more problems successfully with the restriction of the function evaluations. The dimension here is in the range of 20 to 20000. Besides, the algorithms start from the same corresponding initial point $\bx_{\text{int}}$, and the tolerance \(\tau\) is set as \(1\%\) in the profile. We compare \texttt{2D-MoSub} with the methods \texttt{Nelder-Mead} \cite{lagarias1998convergence}, \texttt{NEWUOA} \cite{powell2006newuoa}, \texttt{DFBGN} \cite{cartis2023scalable} and \texttt{CMA-ES} \cite{HansenCMA-ES}.

\begin{figure}[H]
\centering 
\includegraphics[width=0.9\textwidth]{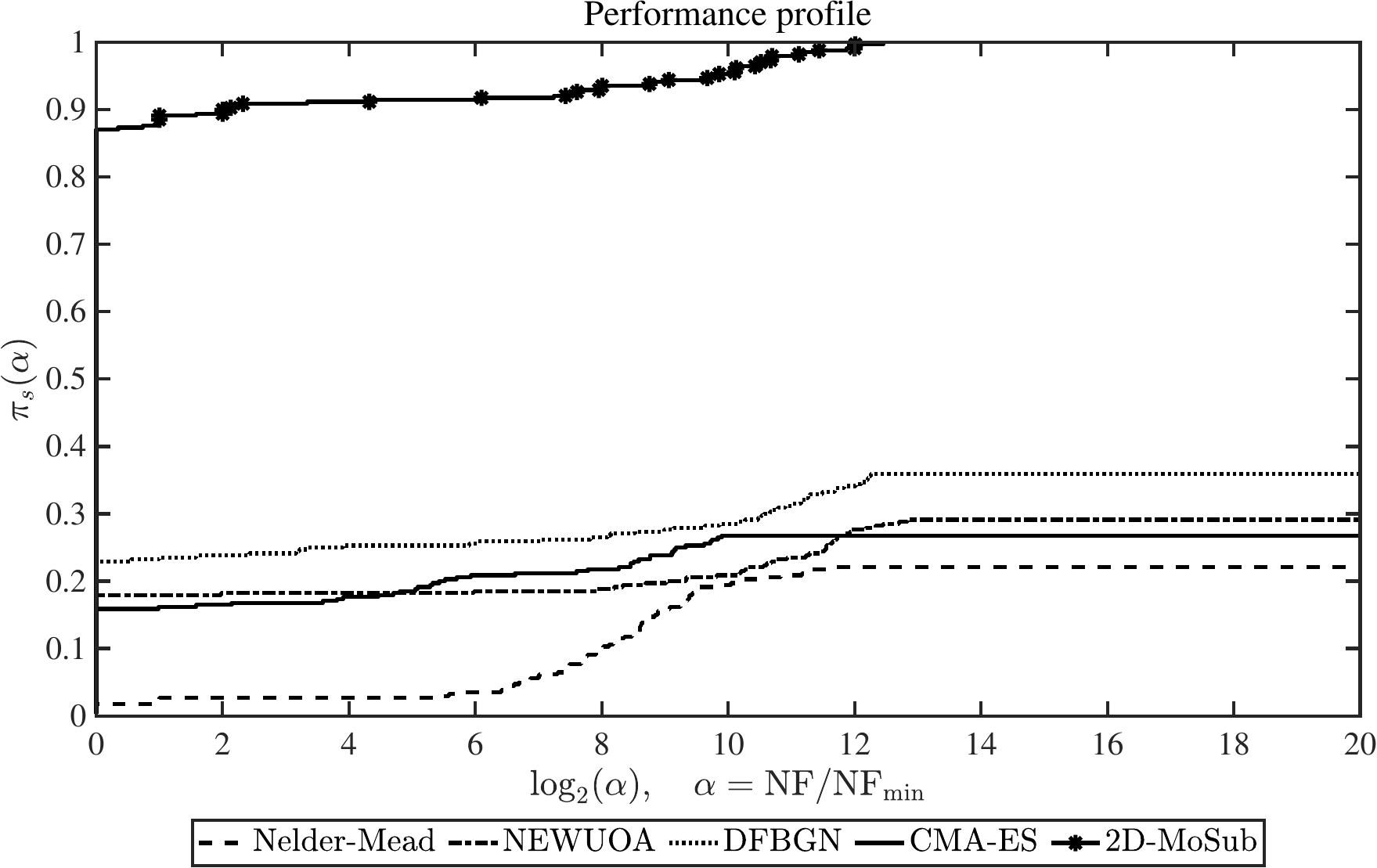} 
\caption{Performance profile of solving test large-scale problems\label{fig-perf}} 
\end{figure}

\begin{figure}[H]
\centering 
\includegraphics[width=0.9\textwidth]{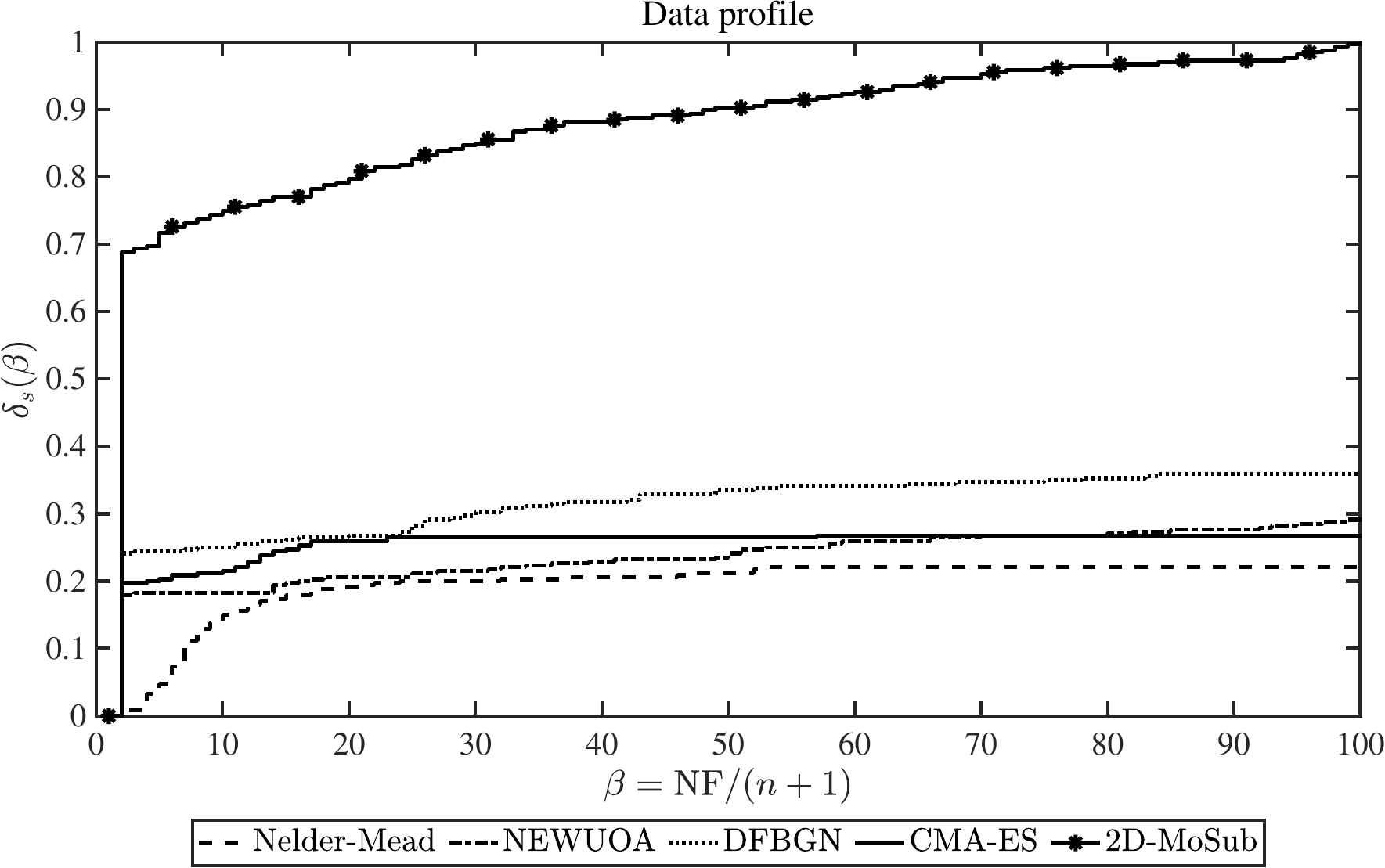}
\caption{Data profile of solving test large-scale problems\label{fig-data}} 
\end{figure}

We can observe from Fig. \ref{fig-perf} and Fig. \ref{fig-data} that \texttt{2D-MoSub} can solve most problems better than the other algorithms. 
{A fact is that our method can rapidly reach an iteration point with an error lower than 1\%-error to the best solution.} 
This shows the obvious advantages of our method.

\section{Conclusion}
We propose the new derivative-free optimization method \texttt{2D-MoSub} based on the trust-region method and subspace techniques for large-scale derivative-free problems. Our method solves 2-dimensional trust-region subproblems during the iteration. Besides, we define the 2-dimensional subspace \(\Lambda\)-poisedness for our interpolation set in the 2-dimensional subspace with 3 determined coefficients at the \(k\)-th step. We give the algorithm's main steps and analyze its theoretical properties. The numerical test shows the numerical advantages of solving derivative-free optimization problems. The future work includes designing a new strategy for choosing the subspace and the research on large-scale constrained problems.


\bibliographystyle{tfs}
\bibliography{interacttfssample}

\end{document}